\documentclass[a4paper,11pt]{article}
\usepackage[T1]{fontenc}
\usepackage{lmodern}
\usepackage[a4paper,margin=2.5cm]{geometry}
\usepackage{amsmath,amssymb,amsthm}
\usepackage{booktabs,array}
\usepackage[hidelinks]{hyperref}
\hypersetup{
  pdftitle={Explicit and logarithmically improved chromatic bounds for (P2 union P3)-free graphs},
  pdfauthor={Lizhong Chen}
}
\allowdisplaybreaks[1]
\newtheorem{theorem}{Theorem}[section]
\newtheorem{lemma}[theorem]{Lemma}
\newtheorem{proposition}[theorem]{Proposition}
\newtheorem{corollary}[theorem]{Corollary}
\theoremstyle{definition}

\newcommand{\Cclass}{\mathcal C}
\newcommand{\Ueight}{U_8}
\newcommand{\Bexp}{B_{\mathrm{exp}}}

\title{Explicit and logarithmically improved chromatic bounds\\
for \((P_2\cup P_3)\)-free graphs}
\author{Lizhong Chen\\[3pt]
\small Department of Mathematics, Hong Kong University of Science and Technology\\
\small Clear Water Bay, Hong Kong\\
\small \href{mailto:lchendh@connect.ust.hk}{lchendh@connect.ust.hk}}

\begin{document}
\maketitle

\begin{abstract}
We prove that every \((P_2\cup P_3)\)-free graph \(G\), with
\(k=\omega(G)\), satisfies
\[
  \chi(G)=O\!\left(k^3\frac{\ln\ln k}{\ln k}\right).
\]
Thus the class admits an \(o(k^3)\) binding function.
We also prove the explicit bound
\[
  \chi(G)\le \frac{k^3}{8}+\frac{5k^2}{4}-4k+8
  \qquad(k\ge8),
\]
with a sharper formula for odd \(k\). A further refinement gives an
explicit cubic bound with leading coefficient \(104/837<1/8\).
\end{abstract}



\section{Introduction}
\label{sec:introduction}

All graphs in this paper are finite, simple and undirected. We write
\(\chi(G)\), \(\omega(G)\) and \(\alpha(G)\) for the chromatic number,
clique number and independence number of a graph \(G\), respectively.
A graph is \(H\)-free if it contains no induced subgraph isomorphic to
\(H\). The path \(P_t\) has \(t\) vertices, and \(P_2\cup P_3\)
consists of an edge and a three-vertex path with no edges between them.
Let \(\Cclass\) be the class of \((P_2\cup P_3)\)-free graphs.

A class of graphs is \(\chi\)-bounded if the chromatic number of every
graph in the class is bounded by a function of its clique number.
Scott and Seymour~\cite{SS2020} survey this theory and its connections
with forbidden induced subgraphs. The ordered maximum-clique approach
used below originates in Wagon's work on graphs with an excluded
induced matching~\cite{Wagon1980}.
For the class \(\Cclass\), Bharathi and Choudum~\cite{BC2018} proved
\begin{equation}\label{eq:classical-bound}
  \chi(G)\le \binom{\omega(G)+2}{3}.
\end{equation}
Their argument assigns vertices to cells indexed by pairs of vertices
of a maximum clique. Each cell is a disjoint union of cliques, and
colouring the cells separately gives the cubic bound. We retain this
decomposition, but colour several cells together.

Stronger bounds have also been studied for subclasses defined by
excluding further induced subgraphs. Prashant, Francis and Francis
Raj~\cite{PFF2024} obtain linear binding functions for several such
subclasses, including families in which a join with a fixed clique
is excluded. Char and Karthick~\cite{CK2025} determine the optimal
binding function when the gem is excluded as well. The gem consists
of a four-vertex path and one vertex adjacent to all four path vertices.
Our estimates apply without an additional forbidden induced subgraph.

Our main result gives a logarithmic saving over cubic growth.

\begin{theorem}\label{thm:logarithmic}
There are absolute constants \(C>0\) and \(k_0\) such that every
\(G\in\Cclass\) with \(k=\omega(G)\ge k_0\) satisfies
\[
  \chi(G)\le Ck^3\frac{\ln\ln k}{\ln k}.
\]
\end{theorem}

For \(k\ge1\), define
\[
  f_{\Cclass}^{=}(k)
  =\sup\{\chi(G):G\in\Cclass,\ \omega(G)=k\}.
\]
The bound~\eqref{eq:classical-bound} makes this quantity finite.
Theorem~\ref{thm:logarithmic} implies
\[
  \lim_{k\to\infty}\frac{f_{\Cclass}^{=}(k)}{k^3}=0.
\]
The result does not give a fixed \(\varepsilon>0\) for which
\(f_{\Cclass}^{=}(k)=O(k^{3-\varepsilon})\). In particular, obtaining
a quadratic binding function requires a further improvement.

We also prove an explicit estimate. For integers \(k\ge8\), put
\begin{equation}\label{eq:parity-saving}
 S(k)=
 \begin{cases}
 (k-5)(k-6)(k-7)/24,&k\text{ odd},\\
 (k-4)(k-6)(k-8)/24,&k\text{ even},
 \end{cases}
 \qquad
 \Ueight(k)=\binom{k+2}{3}-S(k).
\end{equation}

\begin{theorem}\label{thm:one-eighth}
Every \(G\in\Cclass\) with \(k=\omega(G)\ge8\) satisfies
\[
  \chi(G)\le\Ueight(k)\le
  \frac{k^3}{8}+\frac{5k^2}{4}-4k+8.
\]
The second inequality is an equality for even \(k\); for odd \(k\),
the difference between its right-hand side and \(\Ueight(k)\) is
\((k-6)/8\).
\end{theorem}

\begin{corollary}\label{cor:three-twentieths}
Every \(G\in\Cclass\) with clique number \(k\ge47\) satisfies
\[
  \chi(G)\le\frac3{20}k^3.
\]
\end{corollary}

The appendix sharpens the explicit cubic coefficient.

\begin{theorem}\label{thm:refined}
Every \(G\in\Cclass\) with \(k=\omega(G)\ge2\) satisfies
\[
  \chi(G)\le
  \Bexp(k):=\frac{208k^3+3615k^2-871k-1278}{1674}.
\]
For integers \(k\ge8\), one has
\(\Bexp(k)<\Ueight(k)\) if and only if \(k\ge1222\).
\end{theorem}

The coefficient of \(k^3\) in \(\Bexp(k)\) is
\[
  \frac{104}{837}=\frac18-\frac5{6696}.
\]
These explicit estimates and Theorem~\ref{thm:logarithmic} serve different
purposes. The former have specified constants and ranges, whereas the
latter gives the stronger asymptotic rate. We do not optimise the
constant and threshold in the logarithmic estimate.

Both arguments use the ordered decomposition into cells, each a
disjoint union of cliques. Within a column, the nonedges between two
components in different cells form a matching. For a pair of cells,
we share colours along nonedge paths and cycles on two selected
components from each cell, reserving colours for the remaining
components. This gives Theorem~\ref{thm:one-eighth}. To obtain the
logarithmic bound, we control the larger components throughout a column, then partition
its complement into cliques of successively smaller orders. A single
cutoff determines both the retained components and the initial clique
size in this partition.

The logarithmic argument builds on Shearer's entropy
method~\cite{Shearer1995} and the occupancy method of Davies, Jenssen,
Perkins and Roberts~\cite{DJPR2017,DJPR2018}. We use the entropy
inequality and local occupancy theorem of Davies, Kang, Pirot and
Sereni~\cite{DKPS2020}. In the complementary class, the classical
cubic estimate bounds a graph's order polynomially in its clique and
independence numbers. We thus obtain an independent-set estimate whose
dependence on the clique number is explicit. The component-size bounds
then let us apply this estimate at each stage of the clique partition.

For comparison, Molloy's general \(K_r\)-free list-colouring
theorem~\cite[Theorem~2]{Molloy2019} has a factor linear in the excluded
clique order. The estimate in Lemma~\ref{lem:entropy-occupancy} uses the
additional forbidden induced subgraph to obtain logarithmic dependence
on the upper bound for the clique number. This dependence permits
the summation over the clique sizes used in the multiscale argument
in Section~\ref{sec:multiscale}. We state the two probabilistic inputs
and prove the estimates specific to the graph classes considered here.

Section~\ref{sec:decomposition} establishes the decomposition and
elementary estimates. Sections~\ref{sec:pair} and~\ref{sec:explicit}
prove the explicit one-eighth bound. Sections~\ref{sec:entropy}
and~\ref{sec:multiscale} prove the logarithmic bound.
Appendix~\ref{app:refinement} contains the full proof of
Theorem~\ref{thm:refined}.

\section{The ordered decomposition}\label{sec:decomposition}
We begin with the ordered maximum-clique decomposition. It supplies the cells used in both colouring arguments and records the restrictions on edges between them.

For a graph \(H\), write \(N_H(v)\), \(N_H[v]\) and \(\deg_H(v)\) for the open neighbourhood, closed neighbourhood and degree of \(v\). Write \(e(H)=|E(H)|\), \(\overline d(H)=2e(H)/|V(H)|\) for nonempty \(H\), and \(E_H(X,Y)\) for the edges between disjoint sets \(X,Y\). Two disjoint sets are complete to one another if every cross pair is an edge, and anticomplete if no cross pair is an edge. A palette is a set of colour labels. Parameters of a vertex set refer to its induced graph. A \emph{cluster graph} is a disjoint union of cliques; its chromatic number is its largest component order. Empty sets have chromatic, clique and independence numbers zero. Write \(x_+=\max\{x,0\}\), \([m]=\{1,\ldots,m\}\), and \(\theta(L)\) for the minimum number of cliques covering \(V(L)\). Assigning each vertex to one covering clique gives \(\theta(\overline G)=\chi(G)\). All logarithms are natural except \(\log_2\).

The following construction is the ordered decomposition of Bharathi and Choudum~\cite[Section~3]{BC2018}, following Wagon~\cite{Wagon1980}. Fix \(G\in\mathcal C\) and put \(k=\omega(G)\). Choose an ordered maximum clique \(A=(a_1,\ldots,a_k)\). Put an outside vertex in \(I_i\) if its only nonneighbour in \(A\) is \(a_i\); otherwise put it in \(C_{ij}\), where \(i<j\) are its first two nonneighbour indices. The sets \(C_{ij}\) are \emph{cells}, and those with second index \(j\) form a \emph{column} of width \(w_j=k-j+2\).

\begin{lemma}\label{lem:decomposition} The sets \(A,I_i,C_{ij}\) partition \(V(G)\). Each \(\{a_i\}\cup I_i\) is independent. Each cell is a cluster graph with component orders at most its width. The union of two, respectively three, distinct cells of width \(w\) in one column has clique number at most \(w+1\), respectively \(w+2\). Between components of different cells in a column, the nonedges form a matching.

\end{lemma}

\begin{proof} No outside vertex is complete to \(A\). An edge in \(I_i\) extends \(A\setminus\{a_i\}\) to a \((k+1)\)-clique. An induced \(P_3\) in \(C_{ij}\) is anticomplete to the edge \(a_i a_j\), so the cell is a cluster graph. Its vertices are adjacent to every \(a_t\) with \(t<j\), \(t\ne i\); adjoining these \(j-2\) vertices bounds a component's order by \(w_j\). Two or three cells have \(j-3\) or \(j-4\) such common clique neighbours, giving the union bounds.

If \(y\in C_{hj}\) misses two vertices of a component of \(C_{ij}\), their edge is anticomplete to the induced path \(a_j-a_i-y\), a contradiction. Interchanging \(i,h\) proves the matching assertion. \end{proof}

If two cliques have orders \(s,t\) and their cross-nonedges form a matching of size \(m\), their union has clique number \(s+t-m\). In particular, a clique cap \(W\) forces at least \(s+t-W\) matching nonedges. Separate palettes now give the classical bound

\begin{equation}\label{eq:proof-1}
\chi(G)\le k+\sum_{j=2}^k(j-1)(k-j+2)=\binom{k+2}{3}\le(k+2)^3.
\end{equation}

An abstract cell below is any specified cluster part; empty cells are permitted as labelled placeholders. Restricting a graph keeps the cell labels and the intersections with its original components. The matching property survives restriction. Width means a component-order bound, not cell order; the \emph{deficiency} of a width-\(w\) cell \(X\) is \(w-\chi(X)\).

\section{A uniform colouring bound for two cells}\label{sec:pair}
We first seek a bound for the union of two cells that is independent of their numbers of components. Colours can be shared between the two largest components on each side along matching nonedges. The remaining components require pure palettes, and the main issue is to reserve enough colours for them. We construct these palettes first, then use the forbidden induced subgraph to control the tails.

Throughout this section, \(X,Y\) are disjoint cluster cells with matching cross-nonedges between opposite components, and \(W\) is a nonnegative integer clique cap. Components on the same side are anticomplete. A \emph{pure} palette is reserved for one side; colours shared between sides are \emph{mixed}. List the two largest components as \(Q_1,Q_2\subseteq X\) and \(R_1,R_2\subseteq Y\), with orders \(a\ge x\), \(b\ge y\); missing components are empty. Put \(X_0=X\setminus(Q_1\cup Q_2)\) and \(Y_0=Y\setminus(R_1\cup R_2)\).

\subsection{Four components and reserved palettes}

We call a component an \emph{anchor} when its colours or vertex labels are used to colour other components. The following elementary bounds will handle the smaller components and the easier parameter ranges.

\begin{lemma}\label{lem:anchor-bound} If \(\omega(X\cup Y)\le W\), then

\begin{equation}\label{eq:proof-12}
\chi(X\cup Y)\le a+b,\qquad
\chi(X\cup Y)\le W+\min\{x,y\},\qquad
\chi(X\cup Y)\le\max\{W,a+y,b+x\}.
\end{equation}

\end{lemma}

\begin{proof} The first bound uses separate palettes. Colour \(Q_1\) injectively. A \(Y\)-component of order \(s\) has at least \(a+s-W\) matching nonedges to \(Q_1\), so its matched vertices reuse those colours and at most \(W-a\) others need a shared extra palette. Colour \(X\setminus Q_1\) with \(x\) fresh colours, and reverse the sides to obtain the second bound.

For the third, let the \(Q_1,R_1\) nonedge matching have size \(M\), so \(a+b-M\le W\). Give \(t=\min\{M,a-x,b-y\}\) matched pairs distinct mixed colours. Pure palettes of sizes \(a-t\ge x\) and \(b-t\ge y\) cover all other vertices, at total cost
\(a+b-t=\max\{a+b-M,b+x,a+y\}\). \end{proof}

For the remaining configurations, we need to colour four components while reserving colours for both tails. We now allow \(Q_1,Q_2,R_1,R_2\) to be any distinguished components, with \(a,b\) their respective maximum orders. Their union is the \emph{core}; the other components are the \emph{tails}.

\begin{lemma}\label{lem:padding} Suppose \(a,b\le W<a+b\) and every opposite core pair has clique number at most \(W\). Set
\(m=a+b-W\), \(\rho_X=W-b\), \(\rho_Y=W-a\).
By adding dummy vertices and edges, while retaining every original edge, the four components can be made to have orders \(a,a,b,b\), with exactly \(m\) nonedges in each opposite matching.

\end{lemma}

\begin{proof} In a pair of original orders \(s\le a,t\le b\), retain \(\ell=\max\{0,s+t-W\}\) original nonedges. The number still needed, \(m-\ell\), is at most both the total number \((a-s)+(b-t)\) of dummies and \(\min\{a-\ell,b-\ell\}\). Reserve \(u\) first-side and \(v\) second-side dummies with \(u+v=m-\ell\). Match the first group to unreserved second-side vertices and the second to remaining first-side vertices. The displayed capacities guarantee enough vertices. All new nonedges have a dummy endpoint; add every other missing cross-edge. Perform this independently for the four pairs. The resulting graph contains every original edge, so any proper colouring of it restricts to a proper colouring of the original core. \end{proof}

Call the padded components \(A,B,C,D\), with \(A,B\subseteq X\). Their cross-nonedge graph \(\Gamma\) has \(N=2a+2b\) vertices, \(E=4m\) edges and maximum degree at most two. Its nontrivial paths and cycles follow the labels \(A,C,B,D,A,\ldots\), up to reversal and rotation; cycles have length divisible by four. Consecutive pairs or triples, and whole four-cycles, are independent in the graph being coloured.

For a partial colouring let \(l=(l_A,l_B,l_C,l_D)\) count the uncoloured (\emph{loose}) vertices, and define its effective cost as the number of colours plus \(\|l\|_1/2\).

When the loose counts agree on each side and are large enough to cover the tails, completion uses half the total number of loose vertices. The following rounding argument bounds the additional colours needed when the counts are not exactly balanced. We shall use it twice, with different choices of partial colouring.

\begin{lemma}\label{lem:round-complete} Partition the core into pieces, each with two partial colourings. A fractional selection takes a convex combination of the two choices on each piece; its loose vector and effective cost are the corresponding weighted sums. Suppose such a selection has effective cost \(C\) and total loose vector \((R_X,R_X,R_Y,R_Y)\), where \(R_X\ge\chi(X_0)\), \(R_Y\ge\chi(Y_0)\). If, for \(\delta,\varepsilon\ge0\), the two choices in each pair differ by at most \(\delta\) in loose-vector \(\ell_1\)-norm and by at most \(\varepsilon\) in effective cost, the core and tails can be coloured with at most

\begin{equation}\label{eq:proof-13}
C+\delta+2\varepsilon
\end{equation}

colours. The partial colourings on different pieces use disjoint palettes; fixed partial colourings may also be included.

\end{lemma}

\begin{proof} Minimise effective cost on the binary-choice polytope subject to the four loose-count equations. An optimal extreme point has at most four fractional coordinates: five would give a linear dependence allowing a two-sided perturbation. Rounding to nearest endpoints changes the loose vector by at most \(2\delta\) and increases effective cost by at most \(2\varepsilon\).

Complete the colouring with disjoint pure palettes of sizes
\[
 \max\{l_A,l_B,\chi(X_0)\}\qquad\text{and}\qquad
 \max\{l_C,l_D,\chi(Y_0)\}.
\]
For \(R\ge q\ge0\),
\[\max\{u,v,q\}-\frac{u+v}{2}\le\frac{|u-R|+|v-R|}{2}.\]
Thus completion adds at most \(\delta\) to effective cost. Within a side the components are anticomplete; pure palettes are disjoint from all internal colours, so arbitrary edges to the tails are respected. \end{proof}

The following table gives two partial colourings of a path of order \(s\le8\). A block such as \texttt{123} receives one colour, different blocks receive different colours, and the loose columns list uncoloured positions.

\begin{center}\small
\begin{tabular}{@{}c l c l c c@{}}
\toprule
\(s\) & First colouring & Loose & Second colouring & Loose & Common effective cost \\
\midrule
1 & none & 1 & same & 1 & \(1/2\) \\
2 & 12 & none & none & 1, 2 & \(1\) \\
3 & 12 & 3 & 23 & 1 & \(3/2\) \\
4 & 123 & 4 & 234 & 1 & \(3/2\) \\
5 & 234 & 1, 5 & 12, 345 & none & \(2\) \\
6 & 234, 56 & 1 & 12, 345 & 6 & \(5/2\) \\
7 & 234, 567 & 1 & 123, 456 & 7 & \(5/2\) \\
8 & 123, 456, 78 & none & 234, 567 & 1, 8 & \(3\) \\
\bottomrule
\end{tabular}
\end{center}

Each row has mean loose vector equal to half the endpoint-count vector (an isolated vertex has two endpoints), vector difference at most two, and cost at most \(3s/8+3/8\). For a longer path remove successive eight-vertex prefixes, colour each as \texttt{123,456,78}, and apply the table to the last one to eight vertices. The removed prefixes preserve endpoint labels and cost three per eight vertices. Every cycle of order \(s\) costs at most \(3s/8\): use one colour for a four-cycle, eight-vertex blocks otherwise, and one twelve-vertex block of four triples when \(s\equiv4\pmod8\).

\begin{proposition}\label{prop:reserved-tails} Under Lemma~\ref{lem:padding}'s matching and pair-cap assumptions, but allowing \(a+b\le W\), suppose
\(\chi(X_0)\le\min\{a,W-b\}\) and \(\chi(Y_0)\le\min\{b,W-a\}\). Then

\begin{equation}\label{eq:proof-14}
\chi(X\cup Y)\le
\min\left\{a+b,\frac32W+2,
\frac32W-\frac{2W-a-b}{12}+4\right\}.
\end{equation}

\end{proposition}

\begin{proof} Separate palettes give \(a+b\), which proves all three bounds when \(a+b\le W\). Otherwise pad as above. Each label class has total \(\Gamma\)-degree \(2m\), so its endpoint count is \(2\rho_X\) on \(A,B\), and \(2\rho_Y\) on \(C,D\). The table therefore gives mean loose vector \((\rho_X,\rho_X,\rho_Y,\rho_Y)\). If \(p\) counts path components, including isolates, then \(p=N-E\); the total effective cost is at most
\(3(N+p)/8=3W/2\). Lemma~\ref{lem:round-complete} with \((\delta,\varepsilon)=(2,0)\) proves the second bound.

For the improvement, first colour and remove cycles, then eight-vertex prefixes of paths longer than eight. Removing a cycle of order \(4s\) subtracts \(s\) from \(a,b,m,W\) and costs at most \(3s/2\); removing a path prefix subtracts two from those four parameters and costs three. Both preserve \(\rho_X,\rho_Y\). It remains to prove the bound for paths of order at most eight, with unchanged tail capacities.

Let \(g\) count paths not of order three, and let \(h_A,h_B,h_C,h_D\) count three-vertex paths by their middle label; put \(h=\sum h_i\). The table saves at least \(1/8\) from \(3s/8+3/8\) on each of the \(g\) paths. Equal total label counts on each side give
\( |h_A-h_B|,|h_C-h_D|\le g\).
Select three-vertex paths with balanced middle counts on each side, in number
\(h_*=h-|h_A-h_B|-|h_C-h_D|\ge(h-2g)_+\).

On each selected path, choose between colouring the whole triple and leaving all three vertices loose. Their mean effective cost is \(5/4\), saving \(1/4\); their mean loose vector adds half the middle-label unit vector to the ordinary row. The total target remains balanced, with coordinates at least \(\rho_X,\rho_Y\). These new choices have vector difference three and cost difference \(1/2\). Lemma~\ref{lem:round-complete} now gives cost at most
\(3W/2-g/8-h_*/4+4\). Finally, with \(p=g+h=N-E=2(\rho_X+\rho_Y)\),

\[
\frac g8+\frac{h_*}{4}\ge\frac g8+\frac{(p-3g)_+}{4}\ge\frac p{24}.
\]

The function \(g/8+(p-3g)_+/4\) is minimised at \(g=p/3\). Add back the removed pieces and restrict to original vertices to obtain the third bound in \eqref{eq:proof-14}. \end{proof}

\subsection{Tail incidences and exceptions}

The construction above used only the matching condition and bounds on tail sizes. We now use the forbidden induced subgraph to show that the necessary tail bounds hold, apart from a small number of exceptional vertices.

We again take \(Q_1,Q_2,R_1,R_2\) to be the two largest components on each side and assume that the induced graph \(J=G[X\cup Y]\) belongs to \(\mathcal C\) and satisfies \(\omega(J)\le W\).

\begin{lemma}\label{lem:overlap-obstruction} Let \(Q,Q',T\) be distinct first-side components and \(R,D\) distinct second-side components. If \(u\in T,r\in R\) are nonadjacent, at most three vertices of \(D\) have nonneighbours in both \(Q,Q'\); at most two do if \(u\) is complete to \(D\).

\end{lemma}

\begin{proof} Write those nonneighbours of \(z\in D\) as \(q_z,q'_z\); each map is injective. Among four choices, at most one fails each of \(uz\in E(J)\), \(rq_z\in E(J)\), \(rq'_z\in E(J)\). A remaining choice makes \(q_z-r-q'_z\) an induced path anticomplete to \(uz\). If \(u\) is complete to \(D\), only two exclusions remain. \end{proof}

If \(a+b+\min\{x,y\}>2W+3\), the matching domains from \(Q_1,Q_2\) in \(R_1\) overlap in at least \(a+x+b-2W>3\) vertices. Lemma~\ref{lem:overlap-obstruction} therefore makes \(X_0\) complete to \(R_2\). Symmetrically \(Y_0\) is complete to \(Q_2\), and

\begin{equation}\label{eq:proof-15}
\chi(X_0)\le W-y,\qquad \chi(Y_0)\le W-x.
\end{equation}

The next two lemmas handle the exceptional vertices left by the overlap obstruction. In each case, the colouring is constructed so that these vertices can be restored within the stated additive bound.

\begin{lemma}\label{lem:one-sided-exceptions} Let \(R_1,D\) be distinct components of \(Y\), let \(Y_0=Y\setminus(R_1\cup D)\), and put \(a=\chi(X)\). Let \(E\subseteq D\), \(|E|=r\le2\), satisfy \(|D\setminus E|>W/2\), \(\chi(Y_0)\le W-a\), and suppose every vertex of \(D\setminus E\) has at most one nonneighbour in all of \(X\). Then
\(\chi(J)\le W+\lceil(W+r)/2\rceil\le\lfloor3W/2\rfloor+2\).

\end{lemma}

\begin{proof} Set \(d=|D\setminus E|\), \(h=\lceil(W+r)/2\rceil\). If \(a\le h\), separate palettes suffice. Otherwise pad \(R_1\) to order \(W\) so its nonedge matching covers every \(X\)-component: at most \(W-|R_1|\) vertices of each component were unmatched, and each can use a different new dummy, independently across components. Give this anchor \(W\) distinct colours.

From each \(X\)-component of order \(s_i>h\), retain \(l_i=s_i-h\) vertices matched to \(D\setminus E\); there are at least \(s_i+d-W\ge l_i\) available. Give retained vertices their anchor colours and all others a pure \(h\)-colour palette. If \(q\) components retain vertices and \(P=\sum l_i\), their full matching domains in \(D\setminus E\) are disjoint, whence
\(P\le d-q(h+d-W)\). For \(q\ge2\),

\[
P\le2W-2h-d\le W-r-d=W-|D|.
\]

At least \(|D|\) anchor colours are then unused by \(X\) and can colour \(D\). For \(q=1\), the \(P\) distinct vertices of \(D\setminus E\) matched to the retained vertices reuse their colours; the rest use unoccupied anchor colours. For \(q=0\), all anchor colours are available.

At least \(W-a\) anchor colours remain unused by \(X\): for \(q=1\), at most \(a-h\) were used; for \(q\ge2\), at most \(P\le W-|D|<W/2<a\) were used. These colours also cover \(Y_0\), anticomplete to the other \(Y\)-components. The total is \(W+h\), giving the claimed integer bound. \end{proof}

\begin{lemma}\label{lem:two-sided-exceptions} Let the original second-component orders be \(x_0,y_0\). Suppose \(E_X\subseteq Q_2,E_Y\subseteq R_2\) have sizes \(r_X,r_Y\le2\), and put \(x=x_0-r_X\ge W/2\), \(y=y_0-r_Y\ge W/2\). Assume no remaining \(R_2\)-vertex has nonneighbours in both \(Q_1,Q_2\setminus E_X\), and symmetrically for remaining \(Q_2\)-vertices and \(R_1,R_2\setminus E_Y\). If \(a\ge x_0,b\ge y_0\) and the tails have costs at most \(W-y_0,W-x_0\), then
\(\chi(J)\le3W/2+\max\{r_X,r_Y\}\).

\end{lemma}

\begin{proof} Set \(m=x+y-W\), \(A=a-x\), \(B=b-y\), \(d=A+B\). The four matching sizes, ordered as \(Q_1R_1,Q_1(R_2\setminus E_Y),(Q_2\setminus E_X)R_1,(Q_2\setminus E_X)(R_2\setminus E_Y)\), are at least
\(m+d,m+A,m+B,m\). Disjoint domains imply

\begin{equation}\label{eq:proof-16}
2m+A\le y,\qquad2m+B\le x,\qquad3m+d\le W.
\end{equation}

Choose \(m\) nonedges from each middle matching. They block at most \(2m\) edges of the first matching, leaving at least \((d-m)_+\). Choose a further
\(t=\min\{(d-m)_+,A,B\}\) there. Give these \(2m+t\) disjoint pairs mixed colours; pure palettes of sizes \(a-m-t,b-m-t\) cover the remaining core and tails. The cost is \(W+m+d-t\), and the required bound follows from

\begin{equation}\label{eq:proof-17}
m+d-t=\max\{\min(m+d,2m),m+A,m+B\}\le W/2.
\end{equation}

Indeed, the three terms are bounded respectively by \((3m+d)/2\), \(y-m=W-x\), and \(x-m=W-y\), each at most \(W/2\).

Restore exceptions by retaining only \(t'=\min\{t,A-r_X,B-r_Y\}\) of the extra pairs. Then \(0\le t'\) and \(t'\ge t-\max\{r_X,r_Y\}\). The first pure palette has size at least \(x_0-m\), sufficient for the unpaired vertices of the original \(Q_2\), and also for its tail since \(x_0-m=W-y+r_X\ge W-y_0\). The second side is symmetric. Total cost is at most \(3W/2+\max\{r_X,r_Y\}\). \end{proof}

\begin{theorem}\label{thm:universal-pair} If \(X\cup Y\in\mathcal C\) and \(\omega(X\cup Y)\le W\), then

\begin{equation}\label{eq:proof-18}
\chi(X\cup Y)\le\lfloor3W/2\rfloor+2.
\end{equation}

\end{theorem}

\begin{proof} Lemma~\ref{lem:anchor-bound} settles the result unless \(a+b>3W/2+2\) and \(x,y>W/2+2\). Then \eqref{eq:proof-15} holds. If both \(X_0\) is complete to \(R_1\) and \(Y_0\) to \(Q_1\), the tail costs improve to \(W-b,W-a\), so Proposition~\ref{prop:reserved-tails} applies.

If exactly one completeness property fails, say \(ur\) is a nonedge with \(u\in X_0,r\in R_1\), Lemma~\ref{lem:overlap-obstruction} gives at most two vertices of \(R_2\) missing both \(Q_1,Q_2\), because \(u\) is complete to \(R_2\). Delete them as \(E\). Each remaining \(R_2\)-vertex has at most one nonneighbour in \(X\), more than \(W/2\) remain, and \(Y_0\) costs at most \(W-a\). Lemma~\ref{lem:one-sided-exceptions} applies.

If both properties fail, the same obstruction provides exceptional sets of at most two vertices on both second components. Their remaining orders exceed \(W/2\); the disjoint-domain conditions and \eqref{eq:proof-15} give Lemma~\ref{lem:two-sided-exceptions}. In all cases the bound is \(3W/2+2\); take the integer part. \end{proof}

\section{The explicit one-eighth bound}\label{sec:explicit}
\begin{proof}[Proofs of Theorem~\ref{thm:one-eighth} and Corollary~\ref{cor:three-twentieths}]
A pair of width-\(w\) cells in one column has clique cap \(w+1\). By \eqref{eq:proof-18}, its saving over \(2w\) separate colours is at least
\(r_w=\max\{0,\lfloor w/2\rfloor-3\}\), since
\(\lfloor3(w+1)/2\rfloor+\lfloor w/2\rfloor=2w+1\).
Pair cells arbitrarily within every column and use disjoint palettes on the pairs, any unpaired cell, and the \(k\) independent sets. Relative to \eqref{eq:proof-1}, the saving is at least

\begin{equation}\label{eq:proof-19}
\sum_{n=2}^{k-7}\left\lfloor\frac n2\right\rfloor
\left\lfloor\frac{k-n-5}{2}\right\rfloor,
\end{equation}

an empty sum for \(k=8\). With \(L=k-5\), adding zero end terms makes this
\(\sum_{n=0}^L\lfloor n/2\rfloor\lfloor(L-n)/2\rfloor\). For \(L=2t\), its even and odd parts sum to

\[
\sum_{s=0}^t s(t-s)+\sum_{s=0}^{t-1}s(t-1-s)
=\frac{t(t-1)(2t-1)}6;
\]

for \(L=2t+1\), the sum is \(2\sum_{s=0}^t s(t-s)=t(t-1)(t+1)/3\). These are exactly \(S(k)\). Subtracting from \eqref{eq:proof-1} proves \(\chi(G)\le U_8(k)\); expansion gives
\[
 E_8(k):=\frac{k^3}{8}+\frac{5k^2}{4}-4k+8
\]
for even \(k\), and \(U_8(k)=E_8(k)-(k-6)/8\) for odd \(k\).

Finally,
\(3k^3/20-E_8(k)=(k^3-50k^2+160k-320)/40\).
For \(k=47+t\), the numerator is \(t^3+91t^2+2087t+573>0\) for \(t\ge0\). This proves Theorem~\ref{thm:one-eighth} and Corollary~\ref{cor:three-twentieths}.
\end{proof}

\section{Independent-set entropy and column remainders}\label{sec:entropy}
To colour a whole column, we shall partition its complement into cliques. After all cliques of a given order have been removed, the remaining graph has bounded clique number. We need an upper bound on the order of this remainder. An independent set in the complement is a clique in the original column, so its size is bounded by the column's clique number. An independent-set lower bound will therefore control the remainder's order once we have bounded its average degree.

Put \(H_0=\overline{P_2\cup P_3}\). Let \(Z(T)\) count all independent sets of \(T\), including the empty set, and let \(\mu(T)\) be their mean size under the uniform distribution. Complementing \eqref{eq:proof-1} gives

\begin{equation}\label{eq:proof-2}
\theta(T)\le(\alpha(T)+2)^3\qquad(T\text{ is }H_0\text{-free}).
\end{equation}

We use two results of Davies, Kang, Pirot and Sereni at fugacity one, so that all independent sets are equally likely. Their entropy inequality~\cite[equation~(10) in the proof of Lemma~19]{DKPS2020} gives, for every nonempty \(n\)-vertex graph \(T\),
\begin{equation}\label{eq:proof-3}
\ln Z(T)\le\mu(T)\ln\bigl(en/\mu(T)\bigr).
\end{equation}
Their local occupancy theorem~\cite[Theorem~10]{DKPS2020} states that if \(L\ne\varnothing\), \(\Delta(L)\le D\), and \(\beta,\gamma>0\) satisfy
\begin{equation}\label{eq:proof-4}
\frac{\beta}{2Z(F)}+\gamma\mu(F)\ge1
\end{equation}
for every induced subgraph \(F\) of every neighbourhood \(L[N(v)]\), including the empty graph, then
\begin{equation}\label{eq:proof-5}
\alpha(L)\ge\mu(L)\ge\frac{|V(L)|}{\beta+\gamma D}.
\end{equation}
The next lemma verifies this condition using the classical bound for \(\mathcal C\).

\begin{lemma}\label{lem:entropy-occupancy} If \(T\) is \(H_0\)-free, \(\omega(T)\le r\), and \(r\ge1\), then, writing \(z=\ln Z(T)\),

\begin{equation}\label{eq:proof-6}
\mu(T)\ge f_r(z):=\frac{z}{\ln(54er)+3\ln(1+z)},
\end{equation}

where \(f_r\) is increasing. Consequently, for nonempty \(L\) with the same hypotheses and any real \(D\ge e^e\) satisfying \(\Delta(L)\le D\),

\begin{equation}\label{eq:proof-7}
\alpha(L)\ge\frac{|V(L)|\ln D}
 {48D\bigl(\ln(r+1)+\ln\ln D\bigr)}.
\end{equation}

\end{lemma}

\begin{proof} For nonempty \(T\), write \(n=|V(T)|\). Then \(\mu(T)\ge1/2\) and \(\alpha(T)\le z/\ln2\). Hence \eqref{eq:proof-2} gives
\(n\le r(\alpha(T)+2)^3\le27r(1+z)^3\); substituting in \eqref{eq:proof-3} proves \eqref{eq:proof-6}. The derivative of \(f_r\) has positive numerator
\(\ln(54er)+3\ln(1+z)-3z/(1+z)>\ln(54er)-3>0\).
The empty graph gives \(z=\mu(T)=0\).

Set \(h=f_r(\tfrac12\ln D)\), \(\beta=2\sqrt D\), and \(\gamma=1/h\). Every neighbourhood subgraph \(F\) inherits the hypotheses. If \(Z(F)\le\sqrt D\), the first term of \eqref{eq:proof-4} is at least one; otherwise \eqref{eq:proof-6} makes its second term at least one. This includes empty \(F\). Since \(h\le\tfrac12\ln D\le\sqrt D\), \eqref{eq:proof-5} gives

\[
\frac{\alpha(L)}{|V(L)|}\ge\frac h{3D}
=\frac{\ln D}{6D\{\ln(54er)+3\ln(1+\tfrac12\ln D)\}}.
\]

For \(D\ge e^e\), the expression in braces is at most
\(\ln r+8\ln\ln D\le8(\ln(r+1)+\ln\ln D)\), proving \eqref{eq:proof-7}. \end{proof}

We next control the average degree needed in \eqref{eq:proof-7}. A cell may have arbitrarily many components, so its width alone does not bound its order. The forbidden induced subgraph nevertheless restricts the number of complement edges between two cells in terms of their orders and component counts.

\begin{lemma}\label{lem:cross-edge-count} Suppose \(J\in\mathcal C\) is partitioned into cluster cells with matching cross-nonedges between components. If cells \(X,Y\) have orders \(x,y\) and \(p,q\) nonempty components, then

\begin{equation}\label{eq:proof-8}
|E_{\overline J}(X,Y)|\le x+y+4pq.
\end{equation}

\end{lemma}

\begin{proof} Assume \(y\le x\). If every vertex of \(Y\) has at most two complement neighbours in \(X\), the count is at most \(2y\le x+y\). Otherwise choose \(v\in Y\) with complement neighbours \(u_1,u_2,u_3\) in distinct \(X\)-components. Put \(U=N_{\overline J}(v)\cap X\), so \(|U|\le p\), and let \(D_0\) contain \(v\).

In another \(Y\)-component \(D\), at most three vertices miss one of \(u_1,u_2,u_3\) in \(J\). Any other \(z\in D\) has at most one complement neighbour in \(X\setminus U\): two such neighbours \(a,b\) lie in distinct components, and choosing \(u_i\) in neither makes \(a-v-b\) an induced path anticomplete to the edge \(u_i z\). Thus \(D\) sends at most \(|D|+3p\) complement edges to \(X\setminus U\), and at most \(|U|\le p\) to \(U\). There are at most \(x\) complement edges from \(D_0\) to \(X\). Summation proves \eqref{eq:proof-8}; empty cells are immediate. \end{proof}

In a complement with no \(r\)-clique, each cell meets fewer than \(r\) original components. This bounds the component counts in \eqref{eq:proof-8} and allows us to turn \eqref{eq:proof-7} into an order estimate.

\begin{lemma}\label{lem:remainder-order} Let \(J\in\mathcal C\) be partitioned into cluster cells with matching cross-nonedges between components. Suppose that \(J\) has at most \(K\) cells, \(\omega(J)\le K\), and each cell has order at most \(MK\), where \(M\ge1\). If a \(K_r\)-free induced subgraph \(L'\) of \(\overline J\) has order \(N\), then, for all sufficiently large \(K\), uniformly in \(M\ge1\) and integers \(r\ge4\),

\begin{equation}\label{eq:proof-9}
N\le768(M+2)K^2\xi_r+40rK^{3/2}\sqrt{\xi_r},
\qquad \xi_r=\frac{\ln(r+1)+\ln\ln K}{\ln K}.
\end{equation}

\end{lemma}

\begin{proof} Each restricted cell meets at most \(r-1\) components, since their representatives form a complement clique. Internal cell edges contribute at most \(MKN/2\); summing \eqref{eq:proof-8} over cell pairs gives at most \(KN+2K^2r^2\) cross-edges. For \(N>0\),

\begin{equation}\label{eq:proof-10}
\overline d(L')\le(M+2)K+4K^2r^2/N.
\end{equation}

At least \(N/2\) vertices have degree at most twice this average. Their induced graph \(Q\) has \(\alpha(Q)\le K\), \(\omega(Q)\le r\), and maximum degree at most
\(D_*:=\lceil\max\{2\overline d(L'),\sqrt K\}\rceil\).
Apply \eqref{eq:proof-7}. The function \((\ln(r+1)+\ln\ln x)/\ln x\) decreases for \(x\ge\sqrt K\) when \(K\) is large, and its value there is at most \(2\xi_r\). Therefore

\[
N\le192KD_*\xi_r
\le192K(2\overline d(L')+\sqrt K+1)\xi_r.
\]

Using \eqref{eq:proof-10} and \(\sqrt K+1\le2K\) yields
\(N^2\le768(M+2)K^2\xi_r N+1536K^3r^2\xi_r\).
Now \(N^2\le TN+U\) implies \(N\le T+\sqrt U\), and \(\sqrt{1536}<40\). If \(N=0\), \eqref{eq:proof-9} is immediate. \end{proof}

\section{Multiscale colouring of a column}\label{sec:multiscale}
The previous section gives an order bound provided that each cell has bounded order. We now obtain this condition by deleting a small number of whole cells and the smallest components of each remaining cell. The bounds imposed on the larger components will apply simultaneously to every later remainder.

Let \(J\in\mathcal C\) have \(n\ge1\) cluster cells with component orders at most \(w\ge1\), matching cross-nonedges, and \(\omega(J)\le K:=n+w\). In a cell \(X\), list component orders as \(s_1(X)\ge s_2(X)\ge\cdots\), padding with zeros.

\begin{lemma}\label{lem:rank-profile} If \(b\ge3\) is an integer and \(b^2\le K\), fewer than \(b\) cells satisfy \(s_b(X)>5K/b\).

\end{lemma}

\begin{proof} Otherwise choose exactly \(b\) such cells, each with \(b\) components truncated to order \(T=\lfloor5K/b\rfloor+1\). Independently choose one component in each cell. By \eqref{eq:proof-8}, each pair of restricted cells has at most \(2bT+4b^2\) complement edges, so some choice has a complement \(\Gamma\) on \(bT\) vertices with

\[
2e(\Gamma)\le b(b-1)(2T/b+4)\le4(b-1)T,
\]

because \(T\ge2b\). The Caro--Wei inequality~\cite{Caro1979,Wei1981} gives
\(\alpha(\Gamma)\ge |V(\Gamma)|^2/(|V(\Gamma)|+2e(\Gamma))>bT/5>K\), contradicting \(\omega(J)\le K\). For completeness, Caro--Wei follows by selecting every vertex preceding all its neighbours in a uniform random ordering, then applying Cauchy--Schwarz to \(\sum_v1/(\deg(v)+1)\). \end{proof}

\begin{lemma}\label{lem:dyadic-cleanup} For a power of two \(B\ge4\) with \(B^2\le K\), there is an induced core \(J_*\) such that

\begin{equation}\label{eq:proof-11}
\chi(J)\le\chi(J_*)+5K^2/B+2BK.
\end{equation}

Put \(M_r=3+5\log_2r\). Each retained cell has fewer than \(B\) components and order at most \(M_BK\). More generally, every \(K_r\)-free restriction of \(\overline{J_*}\), for dyadic \(4\le r\le B\), has cell orders at most \(M_rK\).

\end{lemma}

\begin{proof} Apply Lemma~\ref{lem:rank-profile} at ranks \(4,8,\ldots,B\). Remove all exceptional cells, fewer than \(4+8+\cdots+B<2B\). Each removed cell is a cluster graph with component orders at most \(w\le K\), so these cells use at most \(2BK\) colours. In every other cell retain the first \(B-1\) components. Its omitted tail is a cluster graph whose largest component has order at most \(s_B(X)\le5K/B\). Since \(n\le K\), all such tails use at most \(5K^2/B\) colours.

The largest three components total at most \(3K\). Each band of ranks \(q,\ldots,2q-1\), for \(q=4,8,\ldots,B/2\), totals at most \(5K\). Thus the largest \(r-1\) components total at most \(M_rK\) for every indicated \(r\). A \(K_r\)-free complement restriction meets at most \(r-1\) components, proving the last assertion even when those components are not a prefix. All deleted pieces use disjoint palettes. \end{proof}

\begin{theorem}\label{thm:column-logarithmic} Uniformly under these column hypotheses,
\(\chi(J)=O(K^2\ln\ln K/\ln K)\).

\end{theorem}

\begin{proof} Set \(u=\ln\ln K\), \(\eta=u/\ln K\), and take the single cutoff
\(B=2^{\lfloor\log_2\ln K\rfloor}\). For large \(K\), \(4\le B\le\ln K<2B\), \(B^2\le K\), and \(M_B=O(u)\). Apply Lemma~\ref{lem:dyadic-cleanup} and put \(L=\overline{J_*}\).

Greedily remove \(B\)-cliques from \(L\), then cliques of orders \(B/2,B/4,\ldots,4\), exhausting each size. Cover the final remainder by singletons. The first stage costs at most \(M_BK^2/B=O(K^2\eta)\). For each dyadic \(r\) with \(4\le r\le B\), consider the remainder after all \(r\)-cliques have been removed. Its complement is an induced subgraph of \(J_*\), with at most \(K\) cells and clique number at most \(K\). Lemma~\ref{lem:dyadic-cleanup} bounds each restricted cell's order by \(M_rK\). Applying Lemma~\ref{lem:remainder-order} with \(M=M_r\) and the same \(K\) bounds the order of the remainder by

\[
N_r=768(M_r+2)K^2\xi_r+40rK^{3/2}\sqrt{\xi_r}.
\]

The remaining stages cost at most \(4\sum_{r=4,8,\ldots,B}N_r/r\): each intermediate stage costs \(2N_r/r\), and the final singleton stage costs \(N_4\). This also covers \(B=4\), when there are no intermediate stages. Throughout this sum \(\xi_r=O(\eta)\), and

\[
\sum_{r=4,8,\ldots}\frac{M_r+2}{r}
=5\sum_{h=2}^{\infty}\frac{h+1}{2^h}=10.
\]

There are \(O(u)\) stages. Combining their costs with \eqref{eq:proof-11} gives

\[
\chi(J)=O\!\left(K^2\eta+K\ln K+K^{3/2}u\sqrt\eta\right)
=O(K^2\eta),
\]

since the last two terms divided by \(K^2\eta\) are \(\ln^2K/(Ku)\) and \(\sqrt{u\ln K/K}\), both tending to zero. All constants and thresholds are independent of \(n,w\). \end{proof}

\begin{proof}[Proof of Theorem~\ref{thm:logarithmic}]
Column \(j\) in Lemma~\ref{lem:decomposition} has \(n=j-1\), \(w=k-j+2\), and hence \(K=k+1\). Colour the \(k-1\) columns separately and add \(k\) colours for \(\{a_i\}\cup I_i\). Theorem~\ref{thm:column-logarithmic} gives

\[
\chi(G)\le k+O\!\left((k-1)(k+1)^2
\frac{\ln\ln(k+1)}{\ln(k+1)}\right)
=O\!\left(k^3\frac{\ln\ln k}{\ln k}\right).
\]

This proves the claimed bound for all sufficiently large \(k\). Bounded clique numbers are covered by \eqref{eq:proof-1}.
\end{proof}

\appendix
\section{A finer explicit cubic bound}\label{app:refinement}
We prove Theorem~\ref{thm:refined} by grouping cells according to their deficiencies. Cells with small deficiencies can be coloured efficiently in triples. For the other cells, the strengthened pair estimate in \eqref{eq:proof-14} gives a saving that depends on the deficiencies. We begin with the three-cell argument.

Take three cells \(X,Y,Z\) from one width-\(w\) column. Their largest components \(Q,R,S\) have orders
\(a=w-e_X,b=w-e_Y,c=w-e_Z\); set \(D=e_X+e_Y+e_Z\) and \(W=w+2\). Pair clique caps are \(w+1\), and the triple cap is \(W\). For a pair of deficiencies \(e,f\) and second-largest component orders \(x,y\), we shall use the condition

\begin{equation}\label{eq:proof-20}
x+y\ge w+\max\{e,f\}+6.
\end{equation}

\subsection{Three components and common labels}

For a pair satisfying \eqref{eq:proof-20}, every later component on either side is complete to the two largest components on the other side. Indeed, the matching domains from first-side components of orders \(w-e,x\) overlap in a second-side component of order \(y\) in at least
\(x+y-w-e-2\ge4\) vertices. Lemma~\ref{lem:overlap-obstruction} excludes a nonedge from a later first-side component to the largest second-side component. Interchanging the two second-side components only enlarges the overlap; reverse the sides for the other assertion. The corresponding tail costs are therefore at most \(f+1,e+1\).

\begin{lemma}\label{lem:third-anchor-conflict} Fix components \(P\subseteq X,T\subseteq Y\), and label a vertex by its unique nonneighbour in \(S\), when it has one. At most \(3(W-c)=3(e_Z+2)\) labels occurring in both \(P\) and \(T\) have adjacent representatives.

\end{lemma}

\begin{proof} For each conflicting label \(i\), write \(p_i\in P\) and \(t_i\in T\) for the two vertices carrying that label. On the set of conflicting labels, join distinct \(i,j\) when \(p_it_j\) or \(p_jt_i\) is a nonedge. This auxiliary graph has maximum degree at most two, so a greedy three-colouring gives an independent set \(I\) of at least one third of the labels. Its \(2|I|\) labelled vertices, together with \(S\) minus the \(|I|\) label vertices, form a clique of order \(c+|I|\le W\). \end{proof}

The complement on one component from each of the three cells has maximum degree at most two; its cycles repeat the three component labels. The following single estimate will control both the largest components and their replacements.

\begin{lemma}\label{lem:triangle-surplus} Let \(T,R,S\) be one component from each cell, with \(|R|=b,|S|=c\), and let the \(R,S\) nonedge matching have size \(p\). In the complement \(L_T\) on their union, at most

\begin{equation}\label{eq:proof-21}
3W-2(b+c)+p
\end{equation}

vertices of \(T\) lie outside triangle components.

\end{lemma}

\begin{proof} The induced graph \(L_T[R\cup S]\) consists of \(p\) matching edges and \(h=b+c-2p\) isolated vertices, so its independence number is \(b+c-p\). Put \(E=\alpha(L_T)-(b+c-p)\ge0\), and call a vertex of \(T\) bad if it lies outside a triangle component of \(L_T\). In a path component with \(t\) vertices from \(T\), \(p_0\) edges of the \(R,S\) matching and \(h_0\) unmatched anchor vertices, the increase in independence number over its induced anchor graph is
\(E_0=\lceil(t+2p_0+h_0)/2\rceil-p_0-h_0=\lceil(t-h_0)/2\rceil\ge0\).
Thus \(t\le2E_0+h_0\). A triangle has no bad vertex. Any other cycle has length \(3t\), \(t\ge2\), and surplus \(E_0=\lfloor t/2\rfloor\), so \(t\le3E_0\). Summing gives at most \(3E+h\) bad vertices. Since \(\alpha(L_T)\le W\), this is at most \eqref{eq:proof-21}. \end{proof}

Let \(\ell\) count triangle components of the complement on \(Q\cup R\cup S\), labelled \(\{q_i,r_i,s_i\}\), \(i\in[\ell]\), and put \(z=w-\ell\). Since \(\ell\le\min\{a,b,c\}\), and Lemma~\ref{lem:triangle-surplus} with \(T=Q\) bounds \(a-\ell\),

\begin{equation}\label{eq:proof-22}
\max\{e_X,e_Y,e_Z\}\le z
\le D+\min\{e_X,e_Y,e_Z\}+6\le\frac43D+6.
\end{equation}

For the middle inequality, \eqref{eq:proof-21} and \(p\le\min\{b,c\}\) first give \(z\le D+\min\{e_Y,e_Z\}+6\); relabel the three sides and take the minimum.

In any component \(T\subseteq X\), delete its vertices outside triangle components of \(L_T\), then vertices in triangles using an \(R,S\) matching edge outside the \(\ell\) common labels. There are at most \(p-\ell\) of the latter. By \eqref{eq:proof-21}, the total deletions per component are at most

\begin{equation}\label{eq:proof-23}
3W-2(b+c)+2p-\ell
\le3W-2\max\{b,c\}-\ell
=R_X:=z+2\min\{e_Y,e_Z\}+6.
\end{equation}

The surviving vertices have distinct common labels \(i\), each nonadjacent to both \(r_i,s_i\). Define \(R_Y,R_Z\) by cyclically permuting the subscripts.

\begin{proposition}\label{prop:three-cell} Writing the three deficiencies in increasing order as \(e_1\le e_2\le e_3\),

\begin{equation}\label{eq:proof-24}
\chi(X\cup Y\cup Z)\le
\begin{cases}
w+\dfrac{23}{3}D+48,&\text{if some pair satisfies~\eqref{eq:proof-20}},\\[2pt]
2w+e_1+2e_2+11\le2w+D+11,&\text{otherwise}.
\end{cases}
\end{equation}

\end{proposition}

\begin{proof} Suppose the pair \(Y,Z\) satisfies \eqref{eq:proof-20}. Keep all \(X\)-components and only the largest two components of \(Y,Z\), trimming each by \eqref{eq:proof-23}. From each \(X\)-component, delete every remaining vertex adjacent to a vertex with the same label in the second \(Y\)- or \(Z\)-component. Lemma~\ref{lem:third-anchor-conflict} bounds these deletions by \(3(e_Z+2)+3(e_Y+2)\). For each adjacent pair with the same label between the second \(Y\)- and \(Z\)-components, delete its \(Y\)-vertex, requiring at most \(3(e_X+2)\) further deletions.

The surviving vertices with any fixed label form an independent set: pairs involving a largest anchor are nonedges by construction, other cross pairs have been treated, and same-cell components are anticomplete. Use \(\ell\) label colours and three pure cell palettes of sizes

\[
p_X=R_X+3(e_Y+e_Z+4),\qquad
p_Y=R_Y+3(e_X+2),\qquad p_Z=R_Z.
\]

These palettes cover deleted vertices, reusing colours across components. They also cover the omitted \(Y,Z\) tails, whose costs are at most \(e_Z+1,e_Y+1\); \eqref{eq:proof-22} makes \(R_Y,R_Z\) at least those values. With
\(m_0=\min(e_X,e_Y)+\min(e_X,e_Z)+\min(e_Y,e_Z)\le D\), the total is

\[
\ell+p_X+p_Y+p_Z=w+2z+2m_0+3D+36
\le w+\frac{23}{3}D+48.
\]

If no pair satisfies \eqref{eq:proof-20}, choose \(R,S\) from the two least deficient cells and let \(X\) be the remaining cell. Colour the \(R,S\) matching pairs together and all unmatched vertices separately, using \(b+c-p\) colours. In each \(X\)-component, vertices in complement triangles reuse the corresponding pair colours; by \eqref{eq:proof-21}, the other vertices need at most \(3W-2(b+c)+p\) fresh colours, reusable across \(X\)-components. This costs \(3W-b-c=w+e_1+e_2+6\). The remaining components in each anchor cell use a separate palette of size equal to its second-largest component order. The sum of these two orders is at most \(w+e_2+5\) because their pair does not satisfy \eqref{eq:proof-20}. This proves the second line of \eqref{eq:proof-24}. \end{proof}

\subsection{Pairing the remaining cells}

We now choose how to group the cells of a column. The sum of a cell's deficiency and its second-largest component order lies in \([0,w]\). Sorting by this quantity will control the total difference between the paired values.

For a same-column pair with deficiencies \(e,f\), largest orders \(w-e,w-f\) and second-largest orders \(x,y\), define \(t_X=e+x,t_Y=f+y\in[0,w]\). A pair satisfying \eqref{eq:proof-20} has tail capacities \(f+1,e+1\), so the strengthened bound \eqref{eq:proof-14}, with \(W=w+1\), gives

\begin{equation}\label{eq:proof-25}
\chi(X\cup Y)\le\frac32w-\frac{e+f}{12}+\frac{16}{3}.
\end{equation}

If \eqref{eq:proof-20} fails, put \(D=e+f\), \(q=\max\{e,f\}\), and \(\Delta=|t_X-t_Y|\). Since \(x+y\le w+q+5\), the last bound of \eqref{eq:proof-12} gives

\begin{equation}\label{eq:proof-26}
\begin{aligned}
\chi(X\cup Y)
&\le\max\{w+1,w+\max(t_X,t_Y)-D\}\\
&\le\max\left\{w+1,\frac32w-\frac12\min(e,f)
                 +\frac12\Delta+\frac52\right\}.
\end{aligned}
\end{equation}

\begin{lemma}\label{lem:refined-column} A column of \(n\) cells of width \(w\ge2\) has chromatic number at most

\begin{equation}\label{eq:proof-27}
n\lambda(w)+\frac{71}{93}w,\qquad
\lambda(w)=\frac{208w+784}{279}.
\end{equation}

\end{lemma}

\begin{proof} Set \(\tau=(5w-160)/93\). Group cells of deficiency below \(\tau\) into triples, leaving at most two; if \(\tau\le0\), this group is empty. Each triple has total deficiency at most \(3\tau\). If some pair in the triple satisfies \eqref{eq:proof-20}, \eqref{eq:proof-24} gives cost
\(w+23\tau+48=3\lambda(w)\). Otherwise its cost is \(2w+3\tau+11\), smaller than \(3\lambda(w)\) by \((7w+241)/93\).

Sort the cells of deficiency at least \(\tau\) by \(t\), and pair consecutive cells. A pair satisfying \eqref{eq:proof-20} costs at most
\(3w/2-\tau/6+16/3=2\lambda(w)\) by \eqref{eq:proof-25}. A pair not satisfying \eqref{eq:proof-20} costs at most \(2\lambda(w)+\Delta/2\) by \eqref{eq:proof-26}. Indeed, this quantity exceeds the second entry of the maximum in \eqref{eq:proof-26} by at least \(\tau/3+17/6>0\), and \(2\lambda(w)>w+1\). The intervals between consecutive paired values have disjoint interiors in \([0,w]\), so their total \(\Delta/2\) is at most \(w/2\).

There are at most three leftover cells. For zero through three leftovers, use respectively \(0,w,3w/2+7/2,5w/2+7/2\) colours, by \eqref{eq:proof-18} with \(W=w+1\). Their excesses over the respective multiples of \(\lambda(w)\) are

\[
0,\quad\frac{71w-784}{279},\quad
\frac{5w-1183}{558},\quad\frac{49w-917}{186},
\]

each at most \(49w/186\). Adding the \(w/2\) contribution from the paired differences gives \(71w/93\). Assign disjoint palettes to all blocks. \end{proof}

\begin{proof}[Proof of the bound in Theorem~\ref{thm:refined}]
For the original columns, \(n=j-1,w=k-j+2\). The sums

\[
\sum_{j=2}^k nw=\frac{k(k-1)(k+4)}6,\quad
\sum_{j=2}^k n=\frac{k(k-1)}2,\quad
\sum_{j=2}^k w=\frac{k(k+1)}2-1
\]

and \eqref{eq:proof-27}, together with \(k\) colours for \(\{a_i\}\cup I_i\), give

\[
\begin{aligned}
\chi(G)&\le k+\frac{208}{279}\frac{k(k-1)(k+4)}6
+\frac{784}{279}\frac{k(k-1)}2
+\frac{71}{93}\left(\frac{k(k+1)}2-1\right)\\
&=\frac{208k^3+3615k^2-871k-1278}{1674}
\qquad(k\ge2).
\end{aligned}
\]

\end{proof}

\subsection{Comparison of the explicit formulas}

\begin{proof}[Proof of the comparison in Theorem~\ref{thm:refined}]

Direct expansion gives

\[
\begin{aligned}
6696(E_8(k)-B_{\rm exp}(k))&=P_e(k):=5k^3-6090k^2-23300k+58680,\\
6696(U_8(k)-B_{\rm exp}(k))&=P_o(k):=5k^3-6090k^2-24137k+63702
\quad(k\text{ odd}).
\end{aligned}
\]

Each derivative is an upward-opening quadratic with negative constant term, so the polynomial decreases and then increases on the positive axis. The values

\begin{center}\small
\begin{tabular}{@{}c r c r@{}}
\toprule
\(k\) & \(P_e(k)\) & \(k\) & \(P_o(k)\) \\
\midrule
8 & -514920 & 9 & -643176 \\
1220 & -13483320 & 1221 & -7044960 \\
1222 & 1451760 & 1223 & 7937376 \\
\bottomrule
\end{tabular}
\end{center}

therefore imply \(P_e<0\) on \([8,1220]\) and \(P_o<0\) on \([9,1221]\). Both derivatives are positive at \(1222\) and increase thereafter, because their common second derivative is \(30k-12180>0\). Hence \(P_e>0\) for \(k\ge1222\) and \(P_o>0\) for \(k\ge1223\). This covers every even and odd integer \(k\ge8\), proves the exact crossover, and completes the proof of Theorem~\ref{thm:refined}.
\end{proof}

\section*{Declarations}

\noindent\textbf{Funding.}
No specific grant supported this research.

\medskip
\noindent\textbf{Competing interests.}
The author has no competing interests to declare.

\medskip

\noindent\textbf{Use of artificial intelligence.}
OpenAI Codex was used to assist with literature searches, language polishing, and manuscript preparation. Mathematical assistance included work on the local
path-colouring schemes and balance conditions in Section~\ref{sec:pair}, and
the common-label construction for three cells in Appendix~\ref{app:refinement}.
The author takes full responsibility for the contents of the article, including all mathematical claims and proofs.

\end{document}